\documentclass{amsart}
\usepackage{amsfonts}

\newtheorem{theorem}{Theorem}
\theoremstyle{plain}

\newtheorem{corollary}{Corollary}

\newtheorem{lemma}{Lemma}

\newtheorem{remark}{Remark}

\numberwithin{equation}{section}
\input{tcilatex}

\begin{document}
\title[Hypo-$q$-Norms on a Cartesian Product ]{Hypo-$q$-Norms on a Cartesian
Product of Normed Linear Spaces}
\author[S. S. Dragomir]{Silvestru Sever Dragomir$^{1,2}$}
\address{$^{1}$Mathematics, College of Engineering \& Science\\
Victoria University, PO Box 14428\\
Melbourne City, MC 8001, Australia.}
\email{sever.dragomir@vu.edu.au}
\urladdr{http://rgmia.org/dragomir}
\address{$^{2}$DST-NRF Centre of Excellence \\
in the Mathematical and Statistical Sciences, School of Computer Science \&
Applied Mathematics, University of the Witwatersrand, Private Bag 3,
Johannesburg 2050, South Africa}

\begin{abstract}
In this paper we introduce the hypo-$q$-norms on a Cartesian product of
normed linear spaces. A representation of these norms in terms of bounded
linear functionals of norm less than one, the equivalence with the $q$-norms
on a Cartesian product and some reverse inequalities obtained via the scalar
Shisha-Mond, Birnacki et al. and other Gr\"{u}ss type inequalities are also
given.
\end{abstract}

\maketitle

\section{Introduction}

Let $\left( E,\left\Vert \cdot \right\Vert \right) $ be a normed linear
space over the real or complex number field $\mathbb{K}$. On $\mathbb{K}^{n}$
endowed with the canonical linear structure we consider a norm $\left\Vert
\cdot \right\Vert _{n}$ and the unit ball 
\begin{equation*}
\mathbb{B}\left( \left\Vert \cdot \right\Vert _{n}\right) :=\left\{ \mathbf{%
\lambda }=\left( \lambda _{1},\dots ,\lambda _{n}\right) \in \mathbb{K}%
^{n}|\left\Vert \lambda \right\Vert _{n}\leq 1\right\} .
\end{equation*}%
As an example of such norms we should mention the usual $p$-\textit{norms}%
\begin{equation}
\left\Vert \mathbf{\lambda }\right\Vert _{n,p}:=\left\{ 
\begin{array}{ll}
\max \left\{ \left\vert \lambda _{1}\right\vert ,\dots ,\left\vert \lambda
_{n}\right\vert \right\} & \text{if \ }p=\infty ; \\ 
&  \\ 
\left( \sum_{k=1}^{n}\left\vert \lambda _{k}\right\vert ^{p}\right) ^{\frac{1%
}{p}} & \text{if \ }p\in \lbrack 1,\infty ).%
\end{array}%
\right.  \label{e.1.1}
\end{equation}%
The \textit{Euclidean norm} is obtained for $p=2$, i.e.,%
\begin{equation*}
\left\Vert \mathbf{\lambda }\right\Vert _{n,2}=\left(
\sum_{k=1}^{n}\left\vert \lambda _{k}\right\vert ^{2}\right) ^{\frac{1}{2}}.
\end{equation*}%
It is well known that on $E^{n}:=E\times \cdots \times E$ endowed with the
canonical linear structure we can define the following $p$-\textit{norms}:%
\begin{equation}
\left\Vert \mathbf{x}\right\Vert _{n,p}:=\left\{ 
\begin{array}{ll}
\max \left\{ \left\Vert x_{1}\right\Vert ,\dots ,\left\Vert x_{n}\right\Vert
\right\} & \text{if \ }p=\infty ; \\ 
&  \\ 
\left( \sum_{k=1}^{n}\left\Vert x_{k}\right\Vert ^{p}\right) ^{\frac{1}{p}}
& \text{if \ }p\in \lbrack 1,\infty );%
\end{array}%
\right.  \label{e.1.2}
\end{equation}%
where $\mathbf{x}=\left( x_{1},\dots ,x_{n}\right) \in E^{n}.$

Following \cite{SSD1}, for a given norm $\left\Vert \cdot \right\Vert _{n}$
on $\mathbb{K}^{n},$ we define the functional $\left\Vert \cdot \right\Vert
_{h,n}:E^{n}\rightarrow \lbrack 0,\infty )$ given by 
\begin{equation}
\left\Vert \mathbf{x}\right\Vert _{h,n}:=\sup_{\lambda \in B\left(
\left\Vert \cdot \right\Vert _{n}\right) }\left\Vert \sum_{j=1}^{n}\lambda
_{j}x_{j}\right\Vert ,  \label{e.1.3}
\end{equation}%
where $\mathbf{x}=\left( x_{1},\dots ,x_{n}\right) \in E^{n}.$

It is easy to see, by the properties of the norm $\left\Vert \cdot
\right\Vert ,$ that:

\begin{enumerate}
\item[(i)] $\left\Vert \mathbf{x}\right\Vert _{h,n}\geq 0$ for any $\mathbf{x%
}\in E^{n};$

\item[(ii)] $\left\Vert \mathbf{x}+\mathbf{y}\right\Vert _{h,n}\leq
\left\Vert \mathbf{x}\right\Vert _{h,n}+\left\Vert \mathbf{y}\right\Vert
_{h,n}$ for any $\mathbf{x},$ $\mathbf{y}\in E^{n};$

\item[(iii)] $\left\Vert \alpha \mathbf{x}\right\Vert _{h,n}=\left\vert
\alpha \right\vert \left\Vert \mathbf{x}\right\Vert _{h,n}$ for each $\alpha
\in \mathbb{K}$ and $\mathbf{x}\in E^{n};$
\end{enumerate}

and therefore $\left\Vert \cdot \right\Vert _{h,n}$ is a \textit{semi-norm }%
on $E^{n}.$ This will be called the \textit{hypo-semi-norm} generated by the
norm $\left\Vert \cdot \right\Vert _{n}$ on $E^{n}.$

We observe that $\left\Vert \mathbf{x}\right\Vert _{h,n}=0$ if and only if $%
\sum_{j=1}^{n}\lambda _{j}x_{j}=0$ for any $\left( \lambda _{1},\dots
,\lambda _{n}\right) \in B\left( \left\Vert \cdot \right\Vert _{n}\right) .$
If there exists $\lambda _{1}^{0},\dots ,\lambda _{n}^{0}\neq 0$ such that $%
\left( \lambda _{1}^{0},0,\dots ,0\right) ,$ $\left( 0,\lambda
_{2}^{0},\dots ,0\right) ,\dots ,$ $\left( 0,0,\dots ,\lambda
_{n}^{0}\right) \in B\left( \left\Vert \cdot \right\Vert _{n}\right) $ then
the semi-norm generated by $\left\Vert \cdot \right\Vert _{n}$ is a \textit{%
norm} on $E^{n}.$

If by $\mathbb{B}_{n,p}$ with $p\in \left[ 1,\infty \right] $ we denote the
balls generated by the $p$-norms $\left\Vert \cdot \right\Vert _{n,p}$ on $%
\mathbb{K}^{n},$ then we can obtain the following \textit{hypo-}$q$-\textit{%
norms} on $E^{n}:$%
\begin{equation}
\left\Vert \mathbf{x}\right\Vert _{h,n,q}:=\sup_{\mathbf{\lambda }\in 
\mathbb{B}_{n,p}}\left\Vert \sum_{j=1}^{n}\lambda _{j}x_{j}\right\Vert ,
\label{e.1.4}
\end{equation}%
with $q>1$ and $\frac{1}{q}+\frac{1}{p}=1$ if $p>1$, $q=1$ if $p=\infty $
and $q=\infty $ if $p=1.$

For $p=2,$ we have the Euclidean ball in $\mathbb{K}^{n},$ which we denote
by $\mathbb{B}_{n},$ $\mathbb{B}_{n}=\left\{ \mathbf{\lambda }=\left(
\lambda _{1},\dots ,\lambda _{n}\right) \in \mathbb{K}^{n}\left\vert
\sum_{i=1}^{n}\left\vert \lambda _{i}\right\vert ^{2}\right. \leq 1\right\} $
that generates the \textit{hypo-Euclidean norm} on $E^{n},$ i.e.,%
\begin{equation}
\left\Vert \mathbf{x}\right\Vert _{h,e}:=\sup_{\mathbf{\lambda }\in \mathbb{B%
}_{n}}\left\Vert \sum_{j=1}^{n}\lambda _{j}x_{j}\right\Vert .  \label{e.1.5}
\end{equation}

Moreover, if $E=H,$ $H$ is a inner product space over $\mathbb{K}$, then the 
\textit{hypo-Euclidean norm} on $H^{n}$ will be denoted simply by%
\begin{equation}
\left\Vert \mathbf{x}\right\Vert _{e}:=\sup_{\mathbf{\lambda }\in \mathbb{B}%
_{n}}\left\Vert \sum_{j=1}^{n}\lambda _{j}x_{j}\right\Vert .  \label{e.1.6}
\end{equation}

Let $\left( H;\left\langle \cdot ,\cdot \right\rangle \right) $ be a Hilbert
space over $\mathbb{K}$ and $n\in \mathbb{N}$, $n\geq 1.$ In the Cartesian
product $H^{n}:=H\times \cdots \times H,$ for the $n-$tuples of vectors $%
\mathbf{x}=\left( x_{1},\dots ,x_{n}\right) $, $\mathbf{y}=\left(
y_{1},\dots ,y_{n}\right) \in H^{n},$ we can define the inner product $%
\left\langle \cdot ,\cdot \right\rangle $ by%
\begin{equation}
\left\langle \mathbf{x},\mathbf{y}\right\rangle :=\sum_{j=1}^{n}\left\langle
x_{j},y_{j}\right\rangle ,\qquad \mathbf{x},\text{ }\mathbf{y}\in H^{n},
\label{e.1.7}
\end{equation}%
which generates the Euclidean norm $\left\Vert \cdot \right\Vert _{2}$ on $%
H^{n},$ i.e.,%
\begin{equation}
\left\Vert \mathbf{x}\right\Vert _{2}:=\left( \sum_{j=1}^{n}\left\Vert
x_{j}\right\Vert ^{2}\right) ^{\frac{1}{2}},\qquad \mathbf{x}\in H^{n}.
\label{e.1.8}
\end{equation}

The following result established in \cite{SSD1} connects the usual Euclidean
norm $\left\Vert \cdot \right\Vert $ with the hypo-Euclidean norm $%
\left\Vert \cdot \right\Vert _{e}.$

\begin{theorem}[Dragomir, 2007, \protect\cite{SSD1}]
\label{t.A}For any $\mathbf{x}\in H^{n}$ we have the inequalities%
\begin{equation}
\frac{1}{\sqrt{n}}\left\Vert \mathbf{x}\right\Vert \leq \left\Vert \mathbf{x}%
\right\Vert _{e}\leq \left\Vert \mathbf{x}\right\Vert _{2},  \label{e.1.9}
\end{equation}%
i.e., $\left\Vert \cdot \right\Vert _{2}$ and $\left\Vert \cdot \right\Vert
_{e}$ are equivalent norms on $H^{n}.$
\end{theorem}

The following representation result for the hypo-Euclidean norm plays a key
role in obtaining various bounds for this norm:

\begin{theorem}[Dragomir, 2007, \protect\cite{SSD1}]
\label{t.B}For any $\mathbf{x}\in H^{n}$ with $\mathbf{x}=\left( x_{1},\dots
,x_{n}\right) ,$ we have 
\begin{equation}
\left\Vert \mathbf{x}\right\Vert _{e}=\sup_{\left\Vert x\right\Vert \leq
1}\left( \sum_{j=1}^{n}\left\vert \left\langle x,x_{j}\right\rangle
\right\vert ^{2}\right) ^{\frac{1}{2}}.  \label{e.1.10}
\end{equation}
\end{theorem}

Motivated by the above results, in this paper we introduce the hypo-$q$%
-norms on a Cartesian product of normed linear spaces. A representation of
these norms in terms of bounded linear functionals of norm less than one,
the equivalence with the $q$-norms on a Cartesian product and some reverse
inequalities obtained via the scalar Shisha-Mond, Birnacki et al. and other
Gr\"{u}ss type inequalities are also given.

\section{General Results}

Let $\left( E,\left\Vert \cdot \right\Vert \right) $ be a normed linear
space over the real or complex number field $\mathbb{K}$. We denote by $%
E^{\ast }$ its dual space endowed with the norm $\left\Vert \cdot
\right\Vert $ defined by 
\begin{equation*}
\left\Vert f\right\Vert :=\sup_{\left\Vert x\right\Vert \leq 1}\left\vert
f\left( x\right) \right\vert <\infty \text{, where }f\in E^{\ast }.
\end{equation*}

We have the following representation result for the \textit{hypo-}$q$-%
\textit{norms} on $E^{n}$.

\begin{theorem}
\label{t.2.1}Let $\left( E,\left\Vert \cdot \right\Vert \right) $ be a
normed linear space over the real or complex number field $\mathbb{K}$. For
any $\mathbf{x}\in E^{n}$ with $\mathbf{x}=\left( x_{1},\dots ,x_{n}\right)
, $ we have%
\begin{equation}
\left\Vert \mathbf{x}\right\Vert _{h,n,q}=\sup_{\left\Vert f\right\Vert \leq
1}\left\{ \left( \sum_{j=1}^{n}\left\vert f\left( x_{j}\right) \right\vert
^{q}\right) ^{1/q}\right\}  \label{e.2.1}
\end{equation}%
where $p,$ $q>1$ with $\frac{1}{p}+\frac{1}{q}=1$, 
\begin{equation}
\left\Vert \mathbf{x}\right\Vert _{h,n,1}=\sup_{\left\Vert f\right\Vert \leq
1}\left\{ \sum_{j=1}^{n}\left\vert f\left( x_{j}\right) \right\vert \right\}
\label{e.2.2}
\end{equation}%
and%
\begin{equation}
\left\Vert \mathbf{x}\right\Vert _{h,n,\infty }=\left\Vert \mathbf{x}%
\right\Vert _{n,\infty }=\max_{j\in \left\{ 1,...,n\right\} }\left\{
\left\Vert x_{j}\right\Vert \right\} .  \label{e.2.2.0}
\end{equation}

In particular,%
\begin{equation}
\left\Vert \mathbf{x}\right\Vert _{h,e}=\sup_{\left\Vert f\right\Vert \leq
1}\left\{ \left( \sum_{j=1}^{n}\left\vert f\left( x_{j}\right) \right\vert
^{2}\right) ^{1/2}\right\} .  \label{e.2.2.a}
\end{equation}
\end{theorem}

\begin{proof}
Using H\"{o}lder's discrete inequality for $p,$ $q>1$ and $\frac{1}{p}+\frac{%
1}{q}=1$ we have 
\begin{equation*}
\left\vert \sum_{j=1}^{n}\alpha _{j}\beta _{j}\right\vert \leq \left(
\sum_{j=1}^{n}\left\vert \alpha _{j}\right\vert ^{p}\right) ^{1/p}\left(
\sum_{j=1}^{n}\left\vert \beta _{j}\right\vert ^{q}\right) ^{1/q},
\end{equation*}%
which implies that 
\begin{equation}
\sup_{\left\Vert \alpha \right\Vert _{p}\leq 1}\left\vert
\sum_{j=1}^{n}\alpha _{j}\beta _{j}\right\vert \leq \left\Vert \beta
\right\Vert _{q}  \label{e.2.3}
\end{equation}%
where $\alpha =\left( \alpha _{1},\dots ,\alpha _{n}\right) $ and $\beta
=\left( \beta _{1},\dots ,\beta _{n}\right) .$

For $\left( \beta _{1},\dots ,\beta _{n}\right) \neq 0,$ consider $\alpha
=\left( \alpha _{1},\dots ,\alpha _{n}\right) $ with%
\begin{equation*}
\alpha _{j}:=\frac{\overline{\beta _{j}}\left\vert \beta _{j}\right\vert
^{q-2}}{\left( \sum_{k=1}^{n}\left\vert \beta _{k}\right\vert ^{q}\right)
^{1/p}}
\end{equation*}%
for those $j$ for which $\beta _{j}\neq 0$ and $\alpha _{j}=0,$ for the rest.

We observe that 
\begin{align*}
\left\vert \sum_{j=1}^{n}\alpha _{j}\beta _{j}\right\vert & =\left\vert
\sum_{j=1}^{n}\frac{\overline{\beta _{j}}\left\vert \beta _{j}\right\vert
^{q-2}}{\left( \sum_{k=1}^{n}\left\vert \beta _{k}\right\vert ^{q}\right)
^{1/p}}\beta _{j}\right\vert =\frac{\sum_{j=1}^{n}\left\vert \beta
_{j}\right\vert ^{q}}{\left( \sum_{k=1}^{n}\left\vert \beta _{k}\right\vert
^{q}\right) ^{1/p}} \\
& =\left( \sum_{j=1}^{n}\left\vert \beta _{j}\right\vert ^{q}\right)
^{1/q}=\left\Vert \beta \right\Vert _{q}
\end{align*}%
and%
\begin{align*}
\left\Vert \alpha \right\Vert _{p}^{p}& =\sum_{j=1}^{n}\left\vert \alpha
_{j}\right\vert ^{p}=\sum_{j=1}^{n}\frac{\left\vert \overline{\beta _{j}}%
\left\vert \beta _{j}\right\vert ^{q-2}\right\vert ^{p}}{\left(
\sum_{k=1}^{n}\left\vert \beta _{k}\right\vert ^{q}\right) }=\sum_{j=1}^{n}%
\frac{\left( \left\vert \beta _{j}\right\vert ^{q-1}\right) ^{p}}{\left(
\sum_{k=1}^{n}\left\vert \beta _{k}\right\vert ^{q}\right) } \\
& =\sum_{j=1}^{n}\frac{\left\vert \beta _{j}\right\vert ^{qp-p}}{\left(
\sum_{k=1}^{n}\left\vert \beta _{k}\right\vert ^{q}\right) }=\sum_{j=1}^{n}%
\frac{\left\vert \beta _{j}\right\vert ^{q}}{\left( \sum_{k=1}^{n}\left\vert
\beta _{k}\right\vert ^{q}\right) }=1.
\end{align*}%
Therefore, by (\ref{e.2.3}) we have the representation%
\begin{equation}
\sup_{\left\Vert \alpha \right\Vert _{p}\leq 1}\left\vert
\sum_{j=1}^{n}\alpha _{j}\beta _{j}\right\vert =\left\Vert \beta \right\Vert
_{q}  \label{e.2.4}
\end{equation}%
for any $\beta =\left( \beta _{1},\dots ,\beta _{n}\right) \in \mathbb{K}%
^{n} $.

By Hahn-Banach theorem, we have for any $u\in E,$ $u\neq 0$ that%
\begin{equation}
\left\Vert u\right\Vert =\sup_{\left\Vert f\right\Vert \leq 1}\left\vert
f\left( u\right) \right\vert .  \label{e.2.4.1}
\end{equation}

Let $\alpha =\left( \alpha _{1},\dots ,\alpha _{n}\right) \in \mathbb{K}^{n}$
and $\mathbf{x}\in E^{n}$ with $\mathbf{x}=\left( x_{1},\dots ,x_{n}\right)
. $ Then by (\ref{e.2.4.1}) we have%
\begin{equation}
\left\Vert \sum_{j=1}^{n}\alpha _{j}x_{j}\right\Vert =\sup_{\left\Vert
f\right\Vert \leq 1}\left\vert f\left( \sum_{j=1}^{n}\alpha _{j}x_{j}\right)
\right\vert =\sup_{\left\Vert f\right\Vert \leq 1}\left\vert
\sum_{j=1}^{n}\alpha _{j}f\left( x_{j}\right) \right\vert .  \label{e.2.4.2}
\end{equation}%
By taking the supremum in this equality we have%
\begin{align*}
\sup_{\left\Vert \alpha \right\Vert _{p}\leq 1}\left\Vert
\sum_{j=1}^{n}\alpha _{j}x_{j}\right\Vert & =\sup_{\left\Vert \alpha
\right\Vert _{p}\leq 1}\left( \sup_{\left\Vert f\right\Vert \leq
1}\left\vert \sum_{j=1}^{n}\alpha _{j}f\left( x_{j}\right) \right\vert
\right) \\
& =\sup_{\left\Vert f\right\Vert \leq 1}\left( \sup_{\left\Vert \alpha
\right\Vert _{p}\leq 1}\left\vert \sum_{j=1}^{n}\alpha _{j}f\left(
x_{j}\right) \right\vert \right) =\sup_{\left\Vert f\right\Vert \leq
1}\left( \sum_{j=1}^{n}\left\vert f\left( x_{j}\right) \right\vert
^{q}\right) ^{1/2},
\end{align*}%
where for the last equality we used the representation (\ref{e.2.4}).

This proves (\ref{e.2.1}).

Using the properties of the modulus, we have%
\begin{equation*}
\left\vert \sum_{j=1}^{n}\alpha _{j}\beta _{j}\right\vert \leq \max_{j\in
\left\{ 1,...,n\right\} }\left\vert \alpha _{j}\right\vert
\sum_{j=1}^{n}\left\vert \beta _{j}\right\vert
\end{equation*}%
which implies that%
\begin{equation}
\sup_{\left\Vert \alpha \right\Vert _{\infty }\leq 1}\left\vert
\sum_{j=1}^{n}\alpha _{j}\beta _{j}\right\vert \leq \left\Vert \beta
\right\Vert _{1}  \label{e.2.5}
\end{equation}%
where $\alpha =\left( \alpha _{1},\dots ,\alpha _{n}\right) $ and $\beta
=\left( \beta _{1},\dots ,\beta _{n}\right) .$

For $\left( \beta _{1},\dots ,\beta _{n}\right) \neq 0,$ consider $\alpha
=\left( \alpha _{1},\dots ,\alpha _{n}\right) $ with $\alpha _{j}:=\frac{%
\overline{\beta _{j}}}{\left\vert \beta _{j}\right\vert }$ for those $j$ for
which $\beta _{j}\neq 0$ and $\alpha _{j}=0,$ for the rest.

We have%
\begin{equation*}
\left\vert \sum_{j=1}^{n}\alpha _{j}\beta _{j}\right\vert =\left\vert
\sum_{j=1}^{n}\frac{\overline{\beta _{j}}}{\left\vert \beta _{j}\right\vert }%
\beta _{j}\right\vert =\sum_{j=1}^{n}\left\vert \beta _{j}\right\vert
=\left\Vert \beta \right\Vert _{1}
\end{equation*}%
and%
\begin{equation*}
\left\Vert \alpha \right\Vert _{\infty }=\max_{j\in \left\{ 1,...,n\right\}
}\left\vert \alpha _{j}\right\vert =\max_{j\in \left\{ 1,...,n\right\}
}\left\vert \frac{\overline{\beta _{j}}}{\left\vert \beta _{j}\right\vert }%
\right\vert =1
\end{equation*}%
and by (\ref{e.2.5}) we get the representation%
\begin{equation}
\sup_{\left\Vert \alpha \right\Vert _{\infty }\leq 1}\left\vert
\sum_{j=1}^{n}\alpha _{j}\beta _{j}\right\vert =\left\Vert \beta \right\Vert
_{1}  \label{e.2.6}
\end{equation}%
for any $\beta =\left( \beta _{1},\dots ,\beta _{n}\right) \in \mathbb{K}%
^{n} $.

By taking the supremum in the equality (\ref{e.2.4.2}) we have 
\begin{align*}
\sup_{\left\Vert \alpha \right\Vert _{\infty }\leq 1}\left\Vert
\sum_{j=1}^{n}\alpha _{j}x_{j}\right\Vert & =\sup_{\left\Vert \alpha
\right\Vert _{\infty }\leq 1}\left( \sup_{\left\Vert f\right\Vert \leq
1}\left\vert \sum_{j=1}^{n}\alpha _{j}f\left( x_{j}\right) \right\vert
\right) \\
& =\sup_{\left\Vert f\right\Vert \leq 1}\left( \sup_{\left\Vert \alpha
\right\Vert _{\infty }\leq 1}\left\vert \sum_{j=1}^{n}\alpha _{j}f\left(
x_{j}\right) \right\vert \right) =\sup_{\left\Vert f\right\Vert \leq
1}\left( \sum_{j=1}^{n}\left\vert f\left( x_{j}\right) \right\vert \right) ,
\end{align*}%
where for the last equality we used the equality (\ref{e.2.6}), which proves
the representation (\ref{e.2.2}).

Finally, we have 
\begin{equation*}
\left\vert \sum_{j=1}^{n}\alpha _{j}\beta _{j}\right\vert \leq
\sum_{j=1}^{n}\left\vert \alpha _{j}\right\vert \max_{j\in \left\{
1,...,n\right\} }\left\vert \beta _{j}\right\vert
\end{equation*}%
which implies that%
\begin{equation}
\sup_{\left\Vert \alpha \right\Vert _{1}\leq 1}\left\vert
\sum_{j=1}^{n}\alpha _{j}\beta _{j}\right\vert \leq \left\Vert \beta
\right\Vert _{\infty }  \label{e.2.7}
\end{equation}%
where $\alpha =\left( \alpha _{1},\dots ,\alpha _{n}\right) $ and $\beta
=\left( \beta _{1},\dots ,\beta _{n}\right) .$

For $\left( \beta _{1},\dots ,\beta _{n}\right) \neq 0$, let $j_{0}\in
\left\{ 1,...,n\right\} $ such that $\left\Vert \beta \right\Vert _{\infty
}=\max_{j\in \left\{ 1,...,n\right\} }\left\vert \beta _{j}\right\vert
=\left\vert \beta _{j_{0}}\right\vert .$ Consider $\alpha =\left( \alpha
_{1},\dots ,\alpha _{n}\right) $ with $\alpha _{j_{0}}=\frac{\overline{\beta
_{j_{0}}}}{\left\vert \beta _{j_{0}}\right\vert }$ and $\alpha _{j}=0$ for $%
j\neq j_{0}.$ For this choice we get%
\begin{equation*}
\sum_{j=1}^{n}\left\vert \alpha _{j}\right\vert =\frac{\left\vert \overline{%
\beta _{j_{0}}}\right\vert }{\left\vert \beta _{j_{0}}\right\vert }=1\text{
and }\left\vert \sum_{j=1}^{n}\alpha _{j}\beta _{j}\right\vert =\left\vert 
\frac{\overline{\beta _{j_{0}}}}{\left\vert \beta _{j_{0}}\right\vert }\beta
_{j_{0}}\right\vert =\left\vert \beta _{j_{0}}\right\vert =\left\Vert \beta
\right\Vert _{\infty },
\end{equation*}%
therefore by (\ref{e.2.7}) we obtain the representation 
\begin{equation}
\sup_{\left\Vert \alpha \right\Vert _{1}\leq 1}\left\vert
\sum_{j=1}^{n}\alpha _{j}\beta _{j}\right\vert =\left\Vert \beta \right\Vert
_{\infty }  \label{e.2.8}
\end{equation}%
for any $\beta =\left( \beta _{1},\dots ,\beta _{n}\right) \in \mathbb{K}%
^{n} $.

By taking the supremum in the equality (\ref{e.2.4.2}) and by using the
equality (\ref{e.2.8}), we have 
\begin{align*}
\sup_{\left\Vert \alpha \right\Vert _{1}\leq 1}\left\Vert
\sum_{j=1}^{n}\alpha _{j}x_{j}\right\Vert & =\sup_{\left\Vert \alpha
\right\Vert _{1}\leq 1}\left( \sup_{\left\Vert f\right\Vert \leq
1}\left\vert \sum_{j=1}^{n}\alpha _{j}f\left( x_{j}\right) \right\vert
\right) \\
& =\sup_{\left\Vert f\right\Vert \leq 1}\left( \sup_{\left\Vert \alpha
\right\Vert _{1}\leq 1}\left\vert \sum_{j=1}^{n}\alpha _{j}f\left(
x_{j}\right) \right\vert \right) =\sup_{\left\Vert f\right\Vert \leq
1}\left( \max_{j\in \left\{ 1,...,n\right\} }\left\vert f\left( x_{j}\right)
\right\vert \right) \\
& =\max_{j\in \left\{ 1,...,n\right\} }\left( \sup_{\left\Vert f\right\Vert
\leq 1}\left\vert f\left( x_{j}\right) \right\vert \right) =\max_{j\in
\left\{ 1,...,n\right\} }\left\{ \left\Vert x_{j}\right\Vert \right\} ,
\end{align*}%
which proves (\ref{e.2.2.0}). For the last equality we used the property (%
\ref{e.2.4.1}).
\end{proof}

\begin{corollary}
\label{c.2.1}With the assumptions of Theorem \ref{t.2.1} we have for $q\geq
1 $ that%
\begin{equation}
\frac{1}{n^{1/q}}\left\Vert \mathbf{x}\right\Vert _{n,q}\leq \left\Vert 
\mathbf{x}\right\Vert _{h,n,q}\leq \left\Vert \mathbf{x}\right\Vert _{n,q}
\label{e.2.9}
\end{equation}%
for any any $\mathbf{x}\in E^{n}.$

In particular, we have 
\begin{equation}
\frac{1}{\sqrt{n}}\left\Vert \mathbf{x}\right\Vert _{2}\leq \left\Vert 
\mathbf{x}\right\Vert _{h,e}\leq \left\Vert \mathbf{x}\right\Vert _{2}
\label{e.2.10}
\end{equation}%
for any any $\mathbf{x}\in E^{n}.$
\end{corollary}

\begin{proof}
Let $\mathbf{x}\in E^{n}$ with $\mathbf{x}=\left( x_{1},\dots ,x_{n}\right) $
and $f\in E^{\ast }$ with $\left\Vert f\right\Vert \leq 1,$ then for $q\geq 1
$%
\begin{equation*}
\left( \sum_{j=1}^{n}\left\vert f\left( x_{j}\right) \right\vert ^{q}\right)
^{1/q}\leq \left( \sum_{j=1}^{n}\left( \left\Vert f\right\Vert \left\Vert
x_{i}\right\Vert \right) ^{q}\right) ^{1/q}=\left\Vert f\right\Vert \left(
\sum_{j=1}^{n}\left\Vert x_{i}\right\Vert ^{q}\right) ^{1/q}=\left\Vert
f\right\Vert \left\Vert \mathbf{x}\right\Vert _{n,q}
\end{equation*}%
and by taking the supremum over $\left\Vert f\right\Vert \leq 1,$ we get the
second inequality in (\ref{e.2.9}).

By the properties of complex numbers, we have 
\begin{equation*}
\max_{j\in \left\{ 1,...,n\right\} }\left\{ \left\vert f\left( x_{j}\right)
\right\vert \right\} \leq \left( \sum_{j=1}^{n}\left\vert f\left(
x_{j}\right) \right\vert ^{q}\right) ^{1/q}
\end{equation*}%
and by taking the supremum over $\left\Vert f\right\Vert \leq 1,$ we get 
\begin{equation}
\sup_{\left\Vert f\right\Vert \leq 1}\left( \max_{j\in \left\{
1,...,n\right\} }\left\{ \left\vert f\left( x_{j}\right) \right\vert
\right\} \right) \leq \sup_{\left\Vert f\right\Vert \leq 1}\left(
\sum_{j=1}^{n}\left\vert f\left( x_{j}\right) \right\vert ^{q}\right) ^{1/q}
\label{e.2.11}
\end{equation}%
and since 
\begin{align*}
\sup_{\left\Vert f\right\Vert \leq 1}\left( \max_{j\in \left\{
1,...,n\right\} }\left\{ \left\vert f\left( x_{j}\right) \right\vert
\right\} \right) & =\max_{j\in \left\{ 1,...,n\right\} }\left\{
\sup_{\left\Vert f\right\Vert \leq 1}\left\vert f\left( x_{j}\right)
\right\vert \right\} \\
& =\max_{j\in \left\{ 1,...,n\right\} }\left\{ \left\Vert x_{j}\right\Vert
\right\} =\left\Vert \mathbf{x}\right\Vert _{n,\infty },
\end{align*}%
then by (\ref{e.2.11}) we get 
\begin{equation}
\left\Vert \mathbf{x}\right\Vert _{n,\infty }\leq \left\Vert \mathbf{x}%
\right\Vert _{h,n,q}\text{ for any }\mathbf{x}\in E^{n}.  \label{e.2.12}
\end{equation}%
Since%
\begin{equation*}
\left( \sum_{j=1}^{n}\left\Vert x_{j}\right\Vert ^{q}\right) ^{1/q}\leq
\left( n\left\Vert \mathbf{x}\right\Vert _{n,\infty }^{q}\right)
^{1/q}=n^{1/q}\left\Vert \mathbf{x}\right\Vert _{n,\infty }
\end{equation*}%
then also 
\begin{equation}
\frac{1}{n^{1/q}}\left\Vert \mathbf{x}\right\Vert _{n,q}\leq \left\Vert 
\mathbf{x}\right\Vert _{n,\infty }\text{ for any }\mathbf{x}\in E^{n}.
\label{e.2.13}
\end{equation}%
By utilising the inequalities (\ref{e.2.12}) and (\ref{e.2.13}) we obtain
the first inequality in (\ref{e.2.9}).
\end{proof}

\begin{remark}
\label{r.2.1}In the case of inner product spaces the inequality (\ref{e.2.10}%
) has been obtained in a different and more difficult way \cite{SSD1} by
employing the rotation-invariant normalised positive Borel measure on the
unit sphere.
\end{remark}

\begin{corollary}
\label{c.2.2}With the assumptions of Theorem \ref{t.2.1} we have for $r\geq
q\geq 1$ that%
\begin{equation}
\left\Vert \mathbf{x}\right\Vert _{h,n,r}\leq \left\Vert \mathbf{x}%
\right\Vert _{h,n,q}\leq n^{\frac{r-q}{rq}}\left\Vert \mathbf{x}\right\Vert
_{h,n,r}  \label{e.2.14}
\end{equation}%
for any any $\mathbf{x}\in E^{n}.$

In particular, for $q\geq 2$ we have%
\begin{equation}
\left\Vert \mathbf{x}\right\Vert _{h,n,q}\leq \left\Vert \mathbf{x}%
\right\Vert _{h,e}\leq n^{\frac{q-2}{2q}}\left\Vert \mathbf{x}\right\Vert
_{h,n,q}  \label{e.2.15}
\end{equation}%
and for $1\leq q\leq 2$ we have%
\begin{equation}
\left\Vert \mathbf{x}\right\Vert _{h,e}\leq \left\Vert \mathbf{x}\right\Vert
_{h,n,q}\leq n^{\frac{2-q}{2q}}\left\Vert \mathbf{x}\right\Vert _{h,e}
\label{e.2.16}
\end{equation}%
for any any $\mathbf{x}\in E^{n}.$
\end{corollary}

\begin{proof}
We use the following elementary inequalities for the nonnegative numbers $%
a_{j}$, $j=1,...,n$ and $r\geq q>0$ (see for instance \cite{SMS}) 
\begin{equation}
\left( \sum_{j=1}^{n}a_{j}^{r}\right) ^{1/r}\leq \left(
\sum_{j=1}^{n}a_{j}^{q}\right) ^{1/q}\leq n^{\frac{r-q}{rq}}\left(
\sum_{j=1}^{n}a_{j}^{r}\right) ^{1/r}.  \label{e.2.17}
\end{equation}%
Let $\mathbf{x}\in E^{n}$ with $\mathbf{x}=\left( x_{1},\dots ,x_{n}\right) $
and $f\in E^{\ast }$ with $\left\Vert f\right\Vert \leq 1,$ then for $r\geq
q\geq 1$ we have 
\begin{equation}
\left( \sum_{j=1}^{n}\left\vert f\left( x_{j}\right) \right\vert ^{r}\right)
^{1/r}\leq \left( \sum_{j=1}^{n}\left\vert f\left( x_{j}\right) \right\vert
^{q}\right) ^{1/q}\leq n^{\frac{r-q}{rq}}\left( \sum_{j=1}^{n}\left\vert
f\left( x_{j}\right) \right\vert ^{r}\right) ^{1/r}.  \label{e.2.18}
\end{equation}%
By taking the supremum over $f\in E^{\ast }$ with $\left\Vert f\right\Vert
\leq 1$ and using Theorem \ref{t.2.1}, we get (\ref{e.2.14}).
\end{proof}

\begin{remark}
\label{r.2.2}If we take $q=1$ in (\ref{e.2.14}), then we get 
\begin{equation}
\left\Vert \mathbf{x}\right\Vert _{h,n,r}\leq \left\Vert \mathbf{x}%
\right\Vert _{h,n,1}\leq n^{\frac{r-1}{r}}\left\Vert \mathbf{x}\right\Vert
_{h,n,r}  \label{e.2.19}
\end{equation}%
for any any $\mathbf{x}\in E^{n}.$

In particular, for $r=2$ we get%
\begin{equation}
\left\Vert \mathbf{x}\right\Vert _{h,e}\leq \left\Vert \mathbf{x}\right\Vert
_{h,n,1}\leq \sqrt{n}\left\Vert \mathbf{x}\right\Vert _{h,e}  \label{e.2.20}
\end{equation}%
for any any $\mathbf{x}\in E^{n}.$
\end{remark}

\section{Some Reverse Inequalities}

Recall the following reverse of Cauchy-Buniakowski-Schwarz inequality \cite%
{SSDRCBS} (see also \cite[Theorem 5. 14]{SSDSurvey})

\begin{lemma}
\label{l.3.1}Let $a,$ $A\in \mathbb{R}$ and $\mathbf{z}=\left( z_{1},\dots
,z_{n}\right) ,$ $\mathbf{y}=\left( y_{1},\dots ,y_{n}\right) $ be two
sequences of real numbers with the property that: 
\begin{equation}
ay_{j}\leq z_{j}\leq Ay_{j}\text{ \hspace{0.05in}for each \hspace{0.05in}}%
j\in \left\{ 1,\dots ,n\right\} .  \label{e.3.1}
\end{equation}%
Then for any $\mathbf{w}=\left( w_{1},\dots ,w_{n}\right) $ a sequence of
positive real numbers, one has the inequality 
\begin{equation}
0\leq \sum_{j=1}^{n}w_{j}z_{j}^{2}\sum_{j=1}^{n}w_{j}y_{j}^{2}-\left(
\sum_{j=1}^{n}w_{j}z_{j}y_{j}\right) ^{2}\leq \frac{1}{4}\left( A-a\right)
^{2}\left( \sum_{j=1}^{n}w_{j}y_{j}^{2}\right) ^{2}.  \label{e.3.2}
\end{equation}%
The constant $\frac{1}{4}$ is sharp in (\ref{e.3.2}).
\end{lemma}

O. Shisha and B. Mond obtained in 1967 (see \cite{SM}) the following
counterparts of $\left( CBS\right) $- inequality (see also \cite[Theorem
5.20 \& 5.21]{SSDSurvey})

\begin{lemma}
\label{l.3.2}Assume that $\mathbf{a}=\left( a_{1},\dots ,a_{n}\right) $ and $%
\mathbf{b}=\left( b_{1},\dots ,b_{n}\right) $ are such that there exists $a,$
$A,$ $b,$ $B$ with the property that: 
\begin{equation}
0\leq a\leq a_{j}\leq A\text{ \hspace{0.05in}and \hspace{0.05in}}0<b\leq
b_{j}\leq B\text{ \hspace{0.05in}for any }j\in \left\{ 1,\dots ,n\right\}
\label{e.3.3}
\end{equation}%
then we have the inequality 
\begin{equation}
\sum_{j=1}^{n}a_{j}^{2}\sum_{j=1}^{n}b_{j}^{2}-\left(
\sum_{j=1}^{n}a_{j}b_{j}\right) ^{2}\leq \left( \sqrt{\frac{A}{b}}-\sqrt{%
\frac{a}{B}}\right) ^{2}\sum_{j=1}^{n}a_{j}b_{j}\sum_{j=1}^{n}b_{j}^{2}.
\label{e.3.4}
\end{equation}
\end{lemma}

and

\begin{lemma}
\label{l.3.3}Assume that $\mathbf{a}$, $\mathbf{b}$ are nonnegative
sequences and there exists $\gamma ,$ $\Gamma $ with the property that 
\begin{equation}
0\leq \gamma \leq \frac{a_{j}}{b_{j}}\leq \Gamma <\infty \text{ \hspace{%
0.05in}for any \hspace{0.05in}}j\in \left\{ 1,\dots ,n\right\} .
\label{e.3.5}
\end{equation}%
Then we have the inequality 
\begin{equation}
0\leq \left( \sum_{j=1}^{n}a_{j}^{2}\sum_{j=1}^{n}b_{j}^{2}\right) ^{\frac{1%
}{2}}-\sum_{j=1}^{n}a_{j}b_{j}\leq \frac{\left( \Gamma -\gamma \right) ^{2}}{%
4\left( \gamma +\Gamma \right) }\sum_{j=1}^{n}b_{j}^{2}.  \label{e.3.6}
\end{equation}
\end{lemma}

We have the following result:

\begin{theorem}
\label{t.3.1}Let $\left( E,\left\Vert \cdot \right\Vert \right) $ be a
normed linear space over the real or complex number field $\mathbb{K}$ and $%
\mathbf{x}\in E^{n}$ with $\mathbf{x}=\left( x_{1},\dots ,x_{n}\right) .$
Then we have%
\begin{equation}
0\leq \left\Vert \mathbf{x}\right\Vert _{h,e}^{2}-\frac{1}{n}\left\Vert 
\mathbf{x}\right\Vert _{h,n,1}^{2}\leq \frac{1}{4}n\left\Vert \mathbf{x}%
\right\Vert _{n,\infty }^{2},  \label{e.3.7}
\end{equation}%
\begin{equation}
0\leq \left\Vert \mathbf{x}\right\Vert _{h,e}^{2}-\frac{1}{n}\left\Vert 
\mathbf{x}\right\Vert _{h,n,1}^{2}\leq \left\Vert \mathbf{x}\right\Vert
_{h,n,1}\left\Vert \mathbf{x}\right\Vert _{n,\infty }  \label{e.3.8}
\end{equation}%
and%
\begin{equation}
0\leq \left\Vert \mathbf{x}\right\Vert _{h,e}-\frac{1}{\sqrt{n}}\left\Vert 
\mathbf{x}\right\Vert _{h,n,1}\leq \frac{1}{4}\sqrt{n}\left\Vert \mathbf{x}%
\right\Vert _{n,\infty }.  \label{e.3.9}
\end{equation}
\end{theorem}

\begin{proof}
Let $\mathbf{x}\in E^{n}$ with $\mathbf{x}=\left( x_{1},\dots ,x_{n}\right) $
and put $R=\max_{j\in \left\{ 1,...,n\right\} }\left\{ \left\Vert
x_{j}\right\Vert \right\} =\left\Vert \mathbf{x}\right\Vert _{n,\infty }.$
If $f\in E^{\ast }$ with $\left\Vert f\right\Vert \leq 1$ then $\left\vert
f\left( x_{j}\right) \right\vert \leq \left\Vert f\right\Vert \left\Vert
x_{j}\right\Vert \leq R$ for any $j\in \left\{ 1,...,n\right\} .$

If we write the inequality (\ref{e.3.2}) for $z_{j}=\left\vert f\left(
x_{j}\right) \right\vert ,$ $w_{j}=y_{j}=1,$ $A=R$ and $a=0,$ we get 
\begin{equation*}
0\leq n\sum_{j=1}^{n}\left\vert f\left( x_{j}\right) \right\vert ^{2}-\left(
\sum_{j=1}^{n}\left\vert f\left( x_{j}\right) \right\vert \right) ^{2}\leq 
\frac{1}{4}n^{2}R^{2}
\end{equation*}%
for any $f\in E^{\ast }$ with $\left\Vert f\right\Vert \leq 1.$

This implies that 
\begin{equation}
\sum_{j=1}^{n}\left\vert f\left( x_{j}\right) \right\vert ^{2}\leq \frac{1}{n%
}\left( \sum_{j=1}^{n}\left\vert f\left( x_{j}\right) \right\vert \right)
^{2}+\frac{1}{4}nR^{2}  \label{e.3.10}
\end{equation}%
for any $f\in E^{\ast }$ with $\left\Vert f\right\Vert \leq 1.$

By taking the supremum in (\ref{e.3.10}) over $f\in E^{\ast }$ with $%
\left\Vert f\right\Vert \leq 1$ we get (\ref{e.3.7}).

If we write the inequality (\ref{e.3.4}) for $a_{j}=\left\vert f\left(
x_{j}\right) \right\vert ,$ $b_{j}=1,$ $b=B=1,$ $a=0$ and $A=R,$ then we get 
\begin{equation*}
0\leq n\sum_{j=1}^{n}\left\vert f\left( x_{j}\right) \right\vert ^{2}-\left(
\sum_{j=1}^{n}\left\vert f\left( x_{j}\right) \right\vert \right) ^{2}\leq
nR\sum_{j=1}^{n}\left\vert f\left( x_{j}\right) \right\vert ,
\end{equation*}%
for any $f\in E^{\ast }$ with $\left\Vert f\right\Vert \leq 1.$

This implies that%
\begin{equation}
\sum_{j=1}^{n}\left\vert f\left( x_{j}\right) \right\vert ^{2}\leq \frac{1}{n%
}\left( \sum_{j=1}^{n}\left\vert f\left( x_{j}\right) \right\vert \right)
^{2}+R\sum_{j=1}^{n}\left\vert f\left( x_{j}\right) \right\vert ,
\label{e.3.11}
\end{equation}%
for any $f\in E^{\ast }$ with $\left\Vert f\right\Vert \leq 1.$

By taking the supremum in (\ref{e.3.11}) over $f\in E^{\ast }$ with $%
\left\Vert f\right\Vert \leq 1$ we get (\ref{e.3.8}).

Finally, if we write the inequality (\ref{e.3.6}) for $a_{j}=\left\vert
f\left( x_{j}\right) \right\vert ,$ $b_{j}=1,$ $b=B=1,$ $\gamma =0$ and $%
\Gamma =R$ we have%
\begin{equation*}
0\leq \left( n\sum_{j=1}^{n}\left\vert f\left( x_{j}\right) \right\vert
^{2}\right) ^{\frac{1}{2}}-\sum_{j=1}^{n}\left\vert f\left( x_{j}\right)
\right\vert \leq \frac{1}{4}nR,
\end{equation*}%
for any $f\in E^{\ast }$ with $\left\Vert f\right\Vert \leq 1.$

This implies that%
\begin{equation}
\left( \sum_{j=1}^{n}\left\vert f\left( x_{j}\right) \right\vert ^{2}\right)
^{\frac{1}{2}}\leq \frac{1}{\sqrt{n}}\sum_{j=1}^{n}\left\vert f\left(
x_{j}\right) \right\vert +\frac{1}{4}\sqrt{n}R,  \label{e.3.12}
\end{equation}%
for any $f\in E^{\ast }$ with $\left\Vert f\right\Vert \leq 1.$

By taking the supremum in (\ref{e.3.12}) over $f\in E^{\ast }$ with $%
\left\Vert f\right\Vert \leq 1$ we get (\ref{e.3.9}).
\end{proof}

Further, we recall the \textit{\v{C}eby\v{s}ev's inequality} for \textit{%
synchronous} $n$-tuples of vectors $\mathbf{a}=\left( a_{1},\dots
,a_{n}\right) $ and $\mathbf{b}=\left( b_{1},\dots ,b_{n}\right) ,$ namely
if $\left( a_{j}-a_{k}\right) \left( b_{j}-b_{k}\right) \geq 0$ for any $j,$ 
$k\in \left\{ 1,...,n\right\} ,$ then%
\begin{equation}
\frac{1}{n}\sum_{j=1}^{n}a_{j}b_{j}\geq \frac{1}{n}\sum_{j=1}^{n}a_{j}\frac{1%
}{n}\sum_{j=1}^{n}b_{j}.  \label{e.3.13}
\end{equation}%
In 1950, Biernacki et al. obtained the following discrete version of Gr\"{u}%
ss' inequality

\begin{lemma}
\label{l.3.4}Assume that $\mathbf{a}=\left( a_{1},\dots ,a_{n}\right) $ and $%
\mathbf{b}=\left( b_{1},\dots ,b_{n}\right) $ are such that there exists
real numbers $a,$ $A,$ $b,$ $B$ with the property that: 
\begin{equation}
a\leq a_{j}\leq A\text{ \hspace{0.05in}and \hspace{0.05in}}b\leq b_{j}\leq B%
\text{ \hspace{0.05in}for any }j\in \left\{ 1,\dots ,n\right\} .
\label{e.3.14}
\end{equation}%
Then%
\begin{align}
& \left\vert \frac{1}{n}\sum_{j=1}^{n}a_{j}b_{j}-\frac{1}{n}%
\sum_{j=1}^{n}a_{j}\frac{1}{n}\sum_{j=1}^{n}b_{j}\right\vert  \label{e.3.15}
\\
& \leq \frac{1}{n}\left\lceil \frac{n}{2}\right\rceil \left( 1-\frac{1}{n}%
\left\lceil \frac{n}{2}\right\rceil \right) \left( A-a\right) \left(
B-b\right)  \notag \\
& =\frac{1}{n^{2}}\left\lceil \frac{n^{2}}{4}\right\rceil \left( A-a\right)
\left( B-a\right) \leq \frac{1}{4}\left( A-a\right) \left( B-b\right) , 
\notag
\end{align}%
where $\left\lceil x\right\rceil $ gives the largest integer less than or
equal to $x.$
\end{lemma}

The following result also holds:

\begin{theorem}
\label{t.3.2}Let $\left( E,\left\Vert \cdot \right\Vert \right) $ be a
normed linear space over the real or complex number field $\mathbb{K}$ and $%
\mathbf{x}\in E^{n}$ with $\mathbf{x}=\left( x_{1},\dots ,x_{n}\right) .$
Then for $q,$ $r\geq 1$ we have%
\begin{eqnarray}
\left\Vert \mathbf{x}\right\Vert _{h,n,q+r}^{q+r} &\leq &\frac{1}{n}%
\left\Vert \mathbf{x}\right\Vert _{h,n,q}^{q}\left\Vert \mathbf{x}%
\right\Vert _{h,n,r}^{r}+\frac{1}{n}\left\lceil \frac{n^{2}}{4}\right\rceil
\left\Vert \mathbf{x}\right\Vert _{n,\infty }^{q+r}  \label{e.3.16} \\
&\leq &\frac{1}{n}\left\Vert \mathbf{x}\right\Vert _{h,n,q}^{q}\left\Vert 
\mathbf{x}\right\Vert _{h,n,r}^{r}+\frac{1}{4}n\left\Vert \mathbf{x}%
\right\Vert _{n,\infty }^{q+r}.  \notag
\end{eqnarray}
\end{theorem}

\begin{proof}
Let $\mathbf{x}\in E^{n}$ with $\mathbf{x}=\left( x_{1},\dots ,x_{n}\right) $
and put $R=\max_{j\in \left\{ 1,...,n\right\} }\left\{ \left\Vert
x_{j}\right\Vert \right\} =\left\Vert \mathbf{x}\right\Vert _{n,\infty }.$
If $f\in E^{\ast }$ with $\left\Vert f\right\Vert \leq 1$ then $\left\vert
f\left( x_{j}\right) \right\vert \leq \left\Vert f\right\Vert \left\Vert
x_{j}\right\Vert \leq R$ for any $j\in \left\{ 1,...,n\right\} .$

If we take into the inequality (\ref{e.3.15}) $a_{j}=\left\vert f\left(
x_{j}\right) \right\vert ^{q}$, $b_{j}=\left\vert f\left( x_{j}\right)
\right\vert ^{r}$, $a=0,$ $A=R^{q},$ $b=0$ and $B=R^{r},$ then we get 
\begin{equation}
\left\vert \frac{1}{n}\sum_{j=1}^{n}\left\vert f\left( x_{j}\right)
\right\vert ^{q+r}-\frac{1}{n}\sum_{j=1}^{n}\left\vert f\left( x_{j}\right)
\right\vert ^{q}\frac{1}{n}\sum_{j=1}^{n}\left\vert f\left( x_{j}\right)
\right\vert ^{r}\right\vert \leq \frac{1}{n^{2}}\left\lceil \frac{n^{2}}{4}%
\right\rceil R^{q+r}.  \label{e.3.17}
\end{equation}%
On the other hand, since the sequences $\left\{ a_{j}\right\} _{j=1,...,n},$ 
$\left\{ b_{j}\right\} _{j=1,...,n}$ are synchronous, then by (\ref{e.3.13})
we have 
\begin{equation*}
0\leq \frac{1}{n}\sum_{j=1}^{n}\left\vert f\left( x_{j}\right) \right\vert
^{q+r}-\frac{1}{n}\sum_{j=1}^{n}\left\vert f\left( x_{j}\right) \right\vert
^{q}\frac{1}{n}\sum_{j=1}^{n}\left\vert f\left( x_{j}\right) \right\vert
^{r}.
\end{equation*}%
Using (\ref{e.3.17}) we then get 
\begin{equation}
\sum_{j=1}^{n}\left\vert f\left( x_{j}\right) \right\vert ^{q+r}\leq \frac{1%
}{n}\sum_{j=1}^{n}\left\vert f\left( x_{j}\right) \right\vert
^{q}\sum_{j=1}^{n}\left\vert f\left( x_{j}\right) \right\vert ^{r}+\frac{1}{n%
}\left\lceil \frac{n^{2}}{4}\right\rceil R^{q+r}  \label{e.3.18}
\end{equation}%
for any $f\in E^{\ast }$ with $\left\Vert f\right\Vert \leq 1.$

By taking the supremum in (\ref{e.3.18}), we get 
\begin{align*}
& \sup_{\left\Vert f\right\Vert \leq 1}\left\{ \sum_{j=1}^{n}\left\vert
f\left( x_{j}\right) \right\vert ^{q+r}\right\} \\
& \leq \frac{1}{n}\sup_{\left\Vert f\right\Vert \leq 1}\left\{
\sum_{j=1}^{n}\left\vert f\left( x_{j}\right) \right\vert
^{q}\sum_{j=1}^{n}\left\vert f\left( x_{j}\right) \right\vert ^{r}\right\} +%
\frac{1}{n}\left\lceil \frac{n^{2}}{4}\right\rceil R^{q+r} \\
& \leq \frac{1}{n}\sup_{\left\Vert f\right\Vert \leq 1}\left\{
\sum_{j=1}^{n}\left\vert f\left( x_{j}\right) \right\vert ^{q}\right\}
\sup_{\left\Vert f\right\Vert \leq 1}\left\{ \sum_{j=1}^{n}\left\vert
f\left( x_{j}\right) \right\vert ^{r}\right\} +\frac{1}{n}\left\lceil \frac{%
n^{2}}{4}\right\rceil R^{q+r},
\end{align*}%
which proves the first inequality in (\ref{e.3.16}).

The second part of (\ref{e.3.16}) is obvious.
\end{proof}

\begin{corollary}
\label{c.3.1}With the assumptions of Theorem \ref{t.3.2} and if $r\geq 1,$
then we have%
\begin{equation}
\left\Vert \mathbf{x}\right\Vert _{h,n,2r}^{2r}\leq \frac{1}{n}\left\Vert 
\mathbf{x}\right\Vert _{h,n,r}^{2r}+\frac{1}{n}\left\lceil \frac{n^{2}}{4}%
\right\rceil \left\Vert \mathbf{x}\right\Vert _{n,\infty }^{2r}\leq \frac{1}{%
n}\left\Vert \mathbf{x}\right\Vert _{h,n,r}^{2r}+\frac{1}{4}n\left\Vert 
\mathbf{x}\right\Vert _{n,\infty }^{2r}.  \label{e.3.19}
\end{equation}%
In particular, for $r=1$ we get%
\begin{equation}
\left\Vert \mathbf{x}\right\Vert _{h,e}^{2}\leq \frac{1}{n}\left\Vert 
\mathbf{x}\right\Vert _{h,n,1}^{2}+\frac{1}{n}\left\lceil \frac{n^{2}}{4}%
\right\rceil \left\Vert \mathbf{x}\right\Vert _{n,\infty }^{2}\leq \frac{1}{n%
}\left\Vert \mathbf{x}\right\Vert _{h,n,1}^{2}+\frac{1}{4}n\left\Vert 
\mathbf{x}\right\Vert _{n,\infty }^{2}.  \label{e.3.20}
\end{equation}
\end{corollary}

The first inequality in (\ref{e.3.20}) is better than the second inequality
in (\ref{e.3.7}).

\section{Reverse Inequalities Via Forward Difference}

For an $n$-tuple of complex numbers $\mathbf{a}=\left( a_{1},\dots
,a_{n}\right) $ with $n\geq 2$ consider the $\left( n-1\right) $-tuple built
by the aid of forward differences $\Delta \mathbf{a=}\left( \Delta
a_{1},...,\Delta a_{n-1}\right) $ where $\Delta a_{k}:=a_{k+1}-a_{k}$ where $%
k\in \left\{ 1,...,n-1\right\} .$ Similarly, if $\mathbf{x}=\left(
x_{1},\dots ,x_{n}\right) \in E^{n}$ is an $n$-tuple of vectors we also can
consider in a similar way the $\left( n-1\right) $-tuple $\Delta \mathbf{x=}%
\left( \Delta x_{1},...,\Delta x_{n-1}\right) .$

We obtained the following Gr\"{u}ss' type inequalities in terms of forward
differences:

\begin{lemma}
\label{l.3.5}Assume that $\mathbf{a}=\left( a_{1},\dots ,a_{n}\right) $ and $%
\mathbf{b}=\left( b_{1},\dots ,b_{n}\right) $ are $n$-tuples of complex
numbers. Then%
\begin{align}
& \left\vert \frac{1}{n}\sum_{j=1}^{n}a_{j}b_{j}-\frac{1}{n}%
\sum_{j=1}^{n}a_{j}\frac{1}{n}\sum_{j=1}^{n}b_{j}\right\vert  \label{e.3.21}
\\
& \leq \left\{ 
\begin{array}{l}
\frac{1}{12}\left( n^{2}-1\right) \left\Vert \Delta \mathbf{a}\right\Vert
_{n-1,\infty }\left\Vert \Delta \mathbf{b}\right\Vert _{n-1,\infty },\text{
\ \cite{DB},} \\ 
\\ 
\frac{1}{6}\frac{n^{2}-1}{n}\left\Vert \Delta \mathbf{a}\right\Vert
_{n-1,\alpha }\left\Vert \Delta \mathbf{b}\right\Vert _{n-1,\beta }\text{
where }\alpha ,\text{ }\beta >1,\text{ }\frac{1}{\alpha }+\frac{1}{\beta }=1,%
\text{ \ \cite{SSDlp},} \\ 
\\ 
\frac{1}{2}\left( 1-\frac{1}{n}\right) \left\Vert \Delta \mathbf{a}%
\right\Vert _{n-1,1}\left\Vert \Delta \mathbf{b}\right\Vert _{n-1,1}\text{,
\ \cite{SSDl1}.}%
\end{array}%
\right.  \notag
\end{align}%
The constants $\frac{1}{12},$ $\frac{1}{6}$ and $\frac{1}{2}$ are best
possible in (\ref{e.3.21}).
\end{lemma}

The following result also holds:

\begin{theorem}
\label{t.3.3}Let $\left( E,\left\Vert \cdot \right\Vert \right) $ be a
normed linear space over the real or complex number field $\mathbb{K}$ and $%
\mathbf{x}\in E^{n}$ with $\mathbf{x}=\left( x_{1},\dots ,x_{n}\right) .$
Then for $q,$ $r\geq 1$ we have%
\begin{align}
\left\Vert \mathbf{x}\right\Vert _{h,n,q+r}^{q+r}& \leq \frac{1}{n}%
\left\Vert \mathbf{x}\right\Vert _{h,n,q}^{q}\left\Vert \mathbf{x}%
\right\Vert _{h,n,r}^{r}  \label{e.3.22} \\
& +\left\{ 
\begin{array}{l}
\frac{1}{12}qr\left( n^{2}-1\right) n\left\Vert \mathbf{x}\right\Vert
_{n,\infty }^{q+r-2}\left\Vert \Delta \mathbf{x}\right\Vert _{n-1,\infty
}^{2}, \\ 
\\ 
\text{ }\frac{1}{6}\left( n^{2}-1\right) qr\left\Vert \mathbf{x}\right\Vert
_{n,\infty }^{q+r-2}\left\Vert \Delta \mathbf{x}\right\Vert _{h,n-1,\alpha
}\left\Vert \Delta \mathbf{x}\right\Vert _{h,n-1,\beta }\text{ } \\ 
\text{where }\alpha ,\text{ }\beta >1,\text{ }\frac{1}{\alpha }+\frac{1}{%
\beta }=1,\text{ } \\ 
\\ 
\frac{1}{2}\left( n-1\right) qr\left\Vert \mathbf{x}\right\Vert _{n,\infty
}^{q+r-2}\left\Vert \Delta \mathbf{x}\right\Vert _{h,n-1,1}^{2}.%
\end{array}%
\right.  \notag
\end{align}
\end{theorem}

\begin{proof}
Let $\mathbf{x}\in E^{n}$ with $\mathbf{x}=\left( x_{1},\dots ,x_{n}\right) $
and $f\in E^{\ast }$ with $\left\Vert f\right\Vert \leq 1.$ If we take into
the inequality (\ref{e.3.21}) $a_{j}=\left\vert f\left( x_{j}\right)
\right\vert ^{q}$, $b_{j}=\left\vert f\left( x_{j}\right) \right\vert ^{r}$,
then we get 
\begin{align}
& \left\vert \frac{1}{n}\sum_{j=1}^{n}\left\vert f\left( x_{j}\right)
\right\vert ^{q+r}-\frac{1}{n}\sum_{j=1}^{n}\left\vert f\left( x_{j}\right)
\right\vert ^{q}\frac{1}{n}\sum_{j=1}^{n}\left\vert f\left( x_{j}\right)
\right\vert ^{r}\right\vert  \label{e.3.23} \\
& \leq \left\{ 
\begin{array}{l}
\frac{1}{12}\left( n^{2}-1\right) \max_{j=1,...,n-1}\left\vert \Delta
\left\vert f\left( x_{j}\right) \right\vert ^{q}\right\vert
\max_{j=1,...,n-1}\left\vert \Delta \left\vert f\left( x_{j}\right)
\right\vert ^{r}\right\vert \text{,} \\ 
\\ 
\frac{1}{6}\frac{n^{2}-1}{n}\left( \sum_{j=1}^{n-1}\left\vert \Delta
\left\vert f\left( x_{j}\right) \right\vert ^{q}\right\vert ^{\alpha
}\right) ^{1/\alpha }\left( \sum_{j=1}^{n-1}\left\vert \Delta \left\vert
f\left( x_{j}\right) \right\vert ^{r}\right\vert ^{\beta }\right) ^{1/\beta }%
\text{ } \\ 
\text{where }\alpha ,\text{ }\beta >1,\text{ }\frac{1}{\alpha }+\frac{1}{%
\beta }=1, \\ 
\\ 
\frac{1}{2}\left( 1-\frac{1}{n}\right) \sum_{j=1}^{n-1}\left\vert \Delta
\left\vert f\left( x_{j}\right) \right\vert ^{q}\right\vert
\sum_{j=1}^{n-1}\left\vert \Delta \left\vert f\left( x_{j}\right)
\right\vert ^{r}\right\vert \text{.}%
\end{array}%
\right.  \notag
\end{align}

We use the following elementary inequality for powers $p\geq 1$%
\begin{equation*}
\left\vert a^{p}-b^{p}\right\vert \leq pR^{p-1}\left\vert a-b\right\vert
\end{equation*}%
where $a,b\in \left[ 0,R\right] .$

Put $R=\max_{j\in \left\{ 1,...,n\right\} }\left\{ \left\Vert
x_{j}\right\Vert \right\} =\left\Vert \mathbf{x}\right\Vert _{n,\infty }.$
Then for any $f\in E^{\ast }$ with $\left\Vert f\right\Vert \leq 1$ we have $%
\left\vert f\left( x_{j}\right) \right\vert \leq \left\Vert f\right\Vert
\left\Vert x_{j}\right\Vert \leq R$ for any $j\in \left\{ 1,...,n\right\} .$

Therefore%
\begin{align}
\left\vert \Delta \left\vert f\left( x_{j}\right) \right\vert
^{q}\right\vert & =\left\vert \left\vert f\left( x_{j+1}\right) \right\vert
^{q}-\left\vert f\left( x_{j}\right) \right\vert ^{q}\right\vert \leq
qR^{q-1}\left\vert \left\vert f\left( x_{j+1}\right) \right\vert -\left\vert
f\left( x_{j}\right) \right\vert \right\vert  \label{e.3.33.a} \\
& \leq qR^{q-1}\left\vert f\left( x_{j+1}\right) -f\left( x_{j}\right)
\right\vert =qR^{q-1}\left\vert f\left( \Delta x_{j}\right) \right\vert 
\notag
\end{align}%
for any $j=1,...,n-1,$ where $\Delta x_{j}=x_{j+1}-x_{j}$ is the forward
difference.

On the other hand, since the sequences $\left\{ a_{j}\right\} _{j=1,...,n},$ 
$\left\{ b_{j}\right\} _{j=1,...,n}$ are synchronous, then we have 
\begin{equation}
0\leq \frac{1}{n}\sum_{j=1}^{n}\left\vert f\left( x_{j}\right) \right\vert
^{q+r}-\frac{1}{n}\sum_{j=1}^{n}\left\vert f\left( x_{j}\right) \right\vert
^{q}\frac{1}{n}\sum_{j=1}^{n}\left\vert f\left( x_{j}\right) \right\vert ^{r}
\label{e.3.33.b}
\end{equation}%
and by the first inequality in (\ref{e.3.23}) we get 
\begin{align}
& \sum_{j=1}^{n}\left\vert f\left( x_{j}\right) \right\vert ^{q+r}
\label{e.3.34} \\
& \leq \frac{1}{n}\sum_{j=1}^{n}\left\vert f\left( x_{j}\right) \right\vert
^{q}\sum_{j=1}^{n}\left\vert f\left( x_{j}\right) \right\vert ^{r}  \notag \\
& +\frac{1}{12}\left( n^{2}-1\right) nqR^{q-1}\max_{j=1,...,n-1}\left\vert
f\left( \Delta x_{j}\right) \right\vert rR^{r-1}\max_{j=1,...,n-1}\left\vert
f\left( \Delta x_{j}\right) \right\vert  \notag \\
& =\frac{1}{n}\sum_{j=1}^{n}\left\vert f\left( x_{j}\right) \right\vert
^{q}\sum_{j=1}^{n}\left\vert f\left( x_{j}\right) \right\vert ^{r}  \notag \\
& +\frac{1}{12}\left( n^{2}-1\right) nqrR^{q+r-2}\left(
\max_{j=1,...,n-1}\left\vert f\left( \Delta x_{j}\right) \right\vert \right)
^{2}  \notag
\end{align}%
for any $f\in E^{\ast }$ with $\left\Vert f\right\Vert \leq 1.$

Taking the supremum over $f\in E^{\ast }$ with $\left\Vert f\right\Vert \leq
1$ in (\ref{e.3.34}) we get the first branch in the inequality (\ref{e.3.22}%
).

We also have, by (\ref{e.3.33.a}), that%
\begin{align*}
\left( \sum_{j=1}^{n-1}\left\vert \Delta \left\vert f\left( x_{j}\right)
\right\vert ^{q}\right\vert ^{\alpha }\right) ^{1/\alpha }& \leq \left[
\left( qR^{q-1}\right) ^{\alpha }\sum_{j=1}^{n-1}\left\vert f\left( \Delta
x_{j}\right) \right\vert ^{\alpha }\right] ^{1/\alpha } \\
& =qR^{q-1}\left( \sum_{j=1}^{n-1}\left\vert f\left( \Delta x_{j}\right)
\right\vert ^{\alpha }\right) ^{1/\alpha }
\end{align*}%
and, similarly,%
\begin{equation*}
\left( \sum_{j=1}^{n-1}\left\vert \Delta \left\vert f\left( x_{j}\right)
\right\vert ^{r}\right\vert ^{\beta }\right) ^{1/\beta }\leq rR^{r-1}\left(
\sum_{j=1}^{n-1}\left\vert f\left( \Delta x_{j}\right) \right\vert ^{\beta
}\right) ^{1/\beta }
\end{equation*}%
where $\alpha ,$ $\beta >1,$ $\frac{1}{\alpha }+\frac{1}{\beta }=1.$

By the second inequality in (\ref{e.3.23}) and by (\ref{e.3.33.b}) we have 
\begin{align}
& \sum_{j=1}^{n}\left\vert f\left( x_{j}\right) \right\vert ^{q+r}
\label{e.3.35} \\
& \leq \frac{1}{n}\sum_{j=1}^{n}\left\vert f\left( x_{j}\right) \right\vert
^{q}\sum_{j=1}^{n}\left\vert f\left( x_{j}\right) \right\vert ^{r}  \notag \\
& +\frac{1}{6}\left( n^{2}-1\right) \left( \sum_{j=1}^{n-1}\left\vert \Delta
\left\vert f\left( x_{j}\right) \right\vert ^{q}\right\vert ^{\alpha
}\right) ^{1/\alpha }\left( \sum_{j=1}^{n-1}\left\vert \Delta \left\vert
f\left( x_{j}\right) \right\vert ^{r}\right\vert ^{\beta }\right) ^{1/\beta }
\notag \\
& \leq \frac{1}{n}\sum_{j=1}^{n}\left\vert f\left( x_{j}\right) \right\vert
^{q}\sum_{j=1}^{n}\left\vert f\left( x_{j}\right) \right\vert ^{r}  \notag \\
& +\frac{1}{6}\left( n^{2}-1\right) qrR^{q+r-2}\left(
\sum_{j=1}^{n-1}\left\vert f\left( \Delta x_{j}\right) \right\vert ^{\alpha
}\right) ^{1/\alpha }\left( \sum_{j=1}^{n-1}\left\vert f\left( \Delta
x_{j}\right) \right\vert ^{\beta }\right) ^{1/\beta }  \notag
\end{align}%
for any $f\in E^{\ast }$ with $\left\Vert f\right\Vert \leq 1,$ where $%
\alpha ,$ $\beta >1,$ $\frac{1}{\alpha }+\frac{1}{\beta }=1.$

Taking the supremum over $f\in E^{\ast }$ with $\left\Vert f\right\Vert \leq
1$ in (\ref{e.3.35}) we get the second branch in the inequality (\ref{e.3.22}%
).

We also have, by (\ref{e.3.33.a}), that%
\begin{equation*}
\sum_{j=1}^{n-1}\left\vert \Delta \left\vert f\left( x_{j}\right)
\right\vert ^{q}\right\vert \leq qR^{q-1}\sum_{j=1}^{n-1}\left\vert f\left(
\Delta x_{j}\right) \right\vert
\end{equation*}%
and%
\begin{equation*}
\sum_{j=1}^{n-1}\left\vert \Delta \left\vert f\left( x_{j}\right)
\right\vert ^{r}\right\vert \leq rR^{r-1}\sum_{j=1}^{n-1}\left\vert f\left(
\Delta x_{j}\right) \right\vert .
\end{equation*}%
By the third inequality in (\ref{e.3.23}) and by (\ref{e.3.33.b}) we have%
\begin{align}
\sum_{j=1}^{n}\left\vert f\left( x_{j}\right) \right\vert ^{q+r}& \leq \frac{%
1}{n}\sum_{j=1}^{n}\left\vert f\left( x_{j}\right) \right\vert
^{q}\sum_{j=1}^{n}\left\vert f\left( x_{j}\right) \right\vert ^{r}
\label{e.3.36} \\
& +\frac{1}{2}\left( n-1\right) \sum_{j=1}^{n-1}\left\vert \Delta \left\vert
f\left( x_{j}\right) \right\vert ^{q}\right\vert \sum_{j=1}^{n-1}\left\vert
\Delta \left\vert f\left( x_{j}\right) \right\vert ^{r}\right\vert  \notag \\
& \leq \frac{1}{n}\sum_{j=1}^{n}\left\vert f\left( x_{j}\right) \right\vert
^{q}\sum_{j=1}^{n}\left\vert f\left( x_{j}\right) \right\vert ^{r}  \notag \\
& +\frac{1}{2}\left( n-1\right) qrR^{q+r-2}\sum_{j=1}^{n-1}\left\vert
f\left( \Delta x_{j}\right) \right\vert \sum_{j=1}^{n-1}\left\vert f\left(
\Delta x_{j}\right) \right\vert  \notag
\end{align}%
for any $f\in E^{\ast }$ with $\left\Vert f\right\Vert \leq 1.$

Taking the supremum over $f\in E^{\ast }$ with $\left\Vert f\right\Vert \leq
1$ in (\ref{e.3.36}) we get the third branch in the inequality (\ref{e.3.22}%
).
\end{proof}

\begin{corollary}
\label{c.3.2}With the assumptions of Theorem \ref{t.3.3} and if $r\geq 1,$
then we have%
\begin{align}
\left\Vert \mathbf{x}\right\Vert _{h,n,2r}^{2r}& \leq \frac{1}{n}\left\Vert 
\mathbf{x}\right\Vert _{h,n,r}^{2r}  \label{e.3.37} \\
& +\left\{ 
\begin{array}{l}
\frac{1}{12}r^{2}\left( n^{2}-1\right) n\left\Vert \mathbf{x}\right\Vert
_{n,\infty }^{2r-2}\left\Vert \Delta \mathbf{x}\right\Vert _{n-1,\infty
}^{2}, \\ 
\\ 
\frac{1}{6}r^{2}\left( n^{2}-1\right) \left\Vert \mathbf{x}\right\Vert
_{n,\infty }^{2r-2}\left\Vert \Delta \mathbf{x}\right\Vert _{h,n-1,\alpha
}\left\Vert \Delta \mathbf{x}\right\Vert _{h,n-1,\beta }\text{ } \\ 
\text{where }\alpha ,\text{ }\beta >1,\text{ }\frac{1}{\alpha }+\frac{1}{%
\beta }=1,\text{ } \\ 
\\ 
\frac{1}{2}r^{2}\left( n-1\right) \left\Vert \mathbf{x}\right\Vert
_{n,\infty }^{2r-2}\left\Vert \Delta \mathbf{x}\right\Vert _{h,n-1,1}^{2}.%
\end{array}%
\right.  \notag
\end{align}%
In particular, for $r=1$ we get%
\begin{equation}
\left\Vert \mathbf{x}\right\Vert _{h,e}^{2}\leq \frac{1}{n}\left\Vert 
\mathbf{x}\right\Vert _{h,n,1}^{2}+\left\{ 
\begin{array}{l}
\frac{1}{12}\left( n^{2}-1\right) n\left\Vert \Delta \mathbf{x}\right\Vert
_{n-1,\infty }^{2}, \\ 
\\ 
\frac{1}{6}\left( n^{2}-1\right) \left\Vert \Delta \mathbf{x}\right\Vert
_{h,n-1,\alpha }\left\Vert \Delta \mathbf{x}\right\Vert _{h,n-1,\beta }\text{
} \\ 
\text{where }\alpha ,\text{ }\beta >1,\text{ }\frac{1}{\alpha }+\frac{1}{%
\beta }=1,\text{ } \\ 
\\ 
\frac{1}{2}\left( n-1\right) \left\Vert \Delta \mathbf{x}\right\Vert
_{h,n-1,1}^{2}.%
\end{array}%
\right. .  \label{e.3.38}
\end{equation}
\end{corollary}

\end{document}